\documentclass[11pt,a4paper]{article}
\title{\bf Nash embedding, shape operator and Navier-Stokes equation on a Riemannian manifold}
\author{Shizan Fang$^1$\footnote{Email: Shizan.Fang@u-bourgogne.fr.}
\vspace{3mm}\\
{\footnotesize $^1$I.M.B, Universit\'e de Bourgogne, BP 47870, 21078 Dijon, France}\\
}
\date{September 20, 2019}
\usepackage{amssymb,amsmath,amsfonts,amsthm,array,bm,color}

\def\R{\mathbb{R}}
\def\E{\mathbb{E}}

\def\T{\mathcal{T}}

\def\X{\mathcal{X}}

\def\A{\mathcal{A}}

\def\L{\mathcal L}
\def\S{\mathbb{S}}

\def\div{\textup{div}}
\def\d{\textup{d}}
\def\Ric{{\rm Ric}}

\def\<{\langle}
\def\>{\rangle}
\def\Def{\rm Def\,}
\let \dis=\displaystyle
\let\ra=\rightarrow

\newtheorem{theorem}{Theorem}[section]

\newtheorem{proposition}[theorem]{Proposition}
\newtheorem{remark}[theorem]{Remark}

\begin{document}

\maketitle
\makeatletter 
\renewcommand\theequation{\thesection.\arabic{equation}}
\@addtoreset{equation}{section}
\makeatother 

\vspace{-8mm}
\begin{abstract} What is the suitable Laplace operator on vector fields for the Navier-Stokes equation on a Riemannian manifold? In this note, by considering Nash embedding, we will try to elucidate different aspects  of different Laplace operators such as de Rham-Hodge Laplacian as well as Ebin-Marsden's Laplacian. A probabilistic representation formula for Navier-Stokes equations on a general compact Riemannian manifold is obtained when de Rham-Hodge Laplacian  is involved.
\end{abstract}


\section{Introduction}

The Navier-Stokes equation on a domain of $\R^n$ or on a torus ${\mathbb T}^n$,
  \begin{equation}\label{NSE}
  \begin{cases}
  \partial_t u+(u\cdot\nabla)u-\nu\Delta u+\nabla p=0,\\
  \nabla\cdot u=0,\quad u|_{t=0}=u_0,
  \end{cases}
  \end{equation}
describe the evolution of the velocity $u$ of an incompressible viscous fluid with kinematic viscosity $\nu>0$, as well as the pressure $p$. Such equation attracts  the attention of many researchers, with an enormous quantity of publications in the literature. The model of periodic boundary conditions is introduced to simplify mathematical considerations. There is no doubt on the importance of the Navier-Stokes equation
on a Riemannian manifold, which is more suitable for models in aerodynamics, meteorology, and so on.
\vskip 2mm

In a seminal paper \cite{EM}, the Navier-Stokes equation has been considered on a compact Riemannian manifold $M$ using the framework of the group of diffeomorphisms of $M$ initiated by V. Arnold  in \cite{Arnold}; the Laplace operator involved in the text of \cite{EM} is de Rham-Hodge Laplacian, however, the authors said in the note added in proof that the convenient Laplace operator comes from deformation tensor. Let's give a brief explanation. Let $\nabla$ be the Levi-Civita connection on $M$, for a vector field $A$ on $M$, the deformation tensor $\Def(A)$ is a symmetric tensor of type $(0,2)$ defined by

\begin{equation}\label{eq1.1}
(\Def A)(X,Y)=\frac{1}{2} \Bigl(\langle \nabla_XA,Y\rangle + \langle \nabla_YA,X\rangle\Bigr),\quad X, Y\in {\mathcal X}(M),
\end{equation}

where ${\mathcal X}(M)$ is the space of vector fields on $M$.  Formula \eqref{eq1.1} says that $\Def A$ is the symmetrized part 
of $\nabla A$. Then $\Def:TM \ra S^2T^*M$ maps a vector field to a symmetric tensor of type 
$(0,2)$. Let $\Def^*: S^2T^*M\ra TM$ be the adjoint operator. Following \cite{MT} or \cite{Taylor}, the Ebin-Marsden's Laplacian 
$\hat\square$ is defined by

\begin{equation}\label{eq1.2}
\hat\square = 2\Def^*\,\Def.
\end{equation}

The following formula holds (see \cite{Taylor} or \cite{ACF}),
\begin{equation}\label{eq1.3}
\hat \square A =-\Delta A - \Ric (A) - \nabla\div (A), \quad A\in\X(M)
\end{equation}
where $\Delta A={\rm Trace} (\nabla^2A)$.
The Ebin-Marsden's Laplacian $\hat\square$ has been used in \cite{ Nagasawa} to solve Navier-Stokes equation on a 
compact Riemannian manifold; recently in \cite{Pier} to the case of Riemannian manifolds 
with negative Ricci curvature : it is quite convenient with sign minus in formula \eqref{eq1.3}. In \cite{Chan}, the authors gave severals arguments from physics in order to convive the relevance of $\hat\square$.

\vskip 2mm
On the other hand, the De Rham-Hodge Laplacian $\square = dd^*+d^*d$ was widely used in the field of Stochastic Analysis 
(see for example \cite{Malliavin, IW, Bismut, Elworthy, ELL, Stroock}). The Navier-Stokes equation with $\square$ on the sphere $S^2$ 
was considered by Temam and Wang in \cite{TemamW} , the case where Ricci tensor is positive. By Bochner-Weitzenb\"ock formula 
(see \cite{Malliavin}):

\begin{equation}\label{eq1.4}
\square  =-\Delta + \Ric ,
\end{equation}

examples of positive Ricci tensor are now confortable with $\square$. Besides, in \cite{Kobayashi}, S. Kobayashi pleaded
for the De Rham-Hodge Laplacian $\square$ in the formalism of the Navier-Stokes equation on manifolds with curvature.
\vskip 2mm

During last decade,  a lot of works have been done in order to establish V. Arnold's variational principle \cite{Arnold} for Navier-Stokes on manifolds (see  for example \cite{AC1, AC2, ACF, Cruzeiro, Luo, ACLZ}).
Connections between Navier-Stokes equations and stochastic evolution also have a quite long history: it can be traced back to a work of Chorin \cite{Chorin}.  In \cite{Flandoli}, a representation formula using noisy flow paths for 3-dimensional Navier-Stokes equation was obtained. An achievement has been realized by Constantin and Iyer in \cite{Constantin} by using stochastic flows. We also refer to \cite{Constantin} for a more complete description on the history of the developments. 
In \cite{ACF}, it was showed that the Ebin-Marsden's Laplacian is naturally involved from the point of view of variational principle; while in \cite{FangLuo}, the De Rham-Hodge Laplacian was showed to be natural for obtaining probabilistic representations.

\section{The case of $\S^n$}\label{sect2}

For reader's convenience, we first introduce some elements in Riemannian manifolds. Let $M$ be a compact Riemannian manifold, 
and $\varphi: M\ra M$ a $C^1$-diffeomorphism; the pull-back $\varphi_*A$ by $\varphi$ of a vector field $A$ is defined by
\begin{equation*}
(\varphi_*A)(x)=d\varphi(\varphi^{-1}(x))A_{\varphi^{-1}(x)},\quad x\in M;
\end{equation*}
the pull-back $\varphi^*\omega$ by $\varphi$ of a differential form $\omega$ is defined by 
\begin{equation*}
\langle \varphi^*\omega, A\rangle_x=\langle \omega_{\varphi(x)},\ d\varphi(x)A_{\varphi(x)}\rangle \quad x\in M.
\end{equation*}
Let $V$ be a vector field on $M$, $\varphi_t$ the group of diffeomorphisms associated to $V$, then the Lie derivative $\L_VA$ of $A$ with respect to $V$ is defined by 

\begin{equation*}
\L_VA=\lim_{t\rightarrow 0}\frac{(\varphi_t^{-1})_*A-A}{t},
\end{equation*}
and the Lie derivative $\L_V\omega$ by
\begin{equation*}
\L_V\omega=\lim_{t\rightarrow 0}\frac{\varphi_t^*\omega-\omega}{t}.
\end{equation*}

We have $\L_VA=[V, A]$. For a vector field $A$ on $M$, the divergence $\div(A)$ of $A$ is defined by 
\begin{equation*}
\int_M \langle\nabla f, A\rangle\, dx=-\int_M f\, \div(A)\, dx,
\end{equation*}
where $dx$ is the Riemannian volume of $M$. If $\nabla$ denotes the Levi-Civita covariant derivative, then 
\begin{equation*}
\div(A)(x)=\sum_{i=1}^n \langle \nabla_{v_i}A, v_i\rangle_{T_xM}
\end{equation*}
for any orthonormal basis  $\{v_1, \ldots, v_n\}$ of $T_xM$. A vector field $A$ on $M$ is said to be of divergence free if $\div(A)=0$. It is important to emphasize that $\L_V A$ is of divergence free if $V$ and $A$ are (see \cite{FangLuo}). 

\vskip 2mm
 In order to understand the geometry of diffusion processes on a manifold $M$, 
 Elworthy, Le Jan and Li  in \cite{ELL}  embedded $M$ into a 
  higher dimensional Riemannian manifold so that the De Rham-Hodge Laplacian $\square$ on differential forms admits suitable decompositions as sum of squares of Lie derivative of a family of vector fields. More precisely, let $\{A_1, \ldots, A_N\}$ be a 
  family of vector fields on $M$, assume they satisfy

  \begin{equation}\label{eq2.1}
  |v|_{T_xM}^2=\sum_{i=1}^N \langle A_i(x), v\rangle^2, \quad v\in T_xM,\ x\in M,
  \end{equation}
   
  \begin{equation}\label{eq2.2}
  \sum_{i=1}^N\nabla_{A_i}A_i=0,
  \end{equation}
  and
  \begin{equation}\label{eq2.3}
  \sum_{i=1}^N A_i\wedge \nabla_{V}A_i=0,\quad\hbox{\rm for any vector field } V.
  \end{equation}
Then they obtained in \cite{ELL}, for any differential form $\omega$ on $M$,

\begin{equation}\label{eq2.4}
  \square \omega=-\sum_{i=1}^N\L_{A_i}^2\omega.
 \end{equation}
 
 There is a one-to-one correspondence between the space of vector fields and that of differential 1-forms. Given a vector field $A$ (resp. differential 1-form $\omega$), we shall denote by $A^\flat$ (resp. $\omega^\sharp$) the corresponding differential 1-form (resp. vector field). The action of the de Rham--Hodge Laplacian $\square$ on the vector field $A$ is defined as follows:
  \begin{equation*}\label{Hodge}
  \square A:=(\square A^\flat)^\sharp.
  \end{equation*}
 
 Surprising enough it was remarked in \cite{FangLuo} that the decomposition \eqref{eq2.4} is in general no more valid if $\omega$ is replaced by a vector field $A$. It is why probabilistic representation formulae for Navier-Stokes equations hold only, until now, on symmetric Riemannian manifolds. 
 
 \begin{proposition}\label{prop2.1} Let $\{A_1, \ldots, A_N\}$ be a family of vector fields satisfying \eqref{eq2.1}--\eqref{eq2.3}. Define
 
 \begin{equation}\label{eq2.5}
 \T_1(B)=\sum_{i=1}^N \div(A_i)\L_{A_i}B,\quad B\in \X(M).
 \end{equation}
 Then $\T_1$ is a tensor : $\X(M)\ra\X(M)$ such that
 \begin{equation}\label{eq2.6}
 \sum_{i=1}^N \L_{A_i}^2B=-\square B - 2 \T_1(B).
 \end{equation}
 \end{proposition}
 
 \begin{proof} Conditions \eqref{eq2.2} and \eqref{eq2.3} imply the following identity (see \cite{ELL, FangLuo})
 \begin{equation}\label{eq2.7}
 \sum_{i=1}^N \div(A_i)A_i=0.
 \end{equation}
 By torsion free, $\L_{A_i}B=[A_i,B]=\nabla_{A_i}B-\nabla_BA_i$. Then according to \eqref{eq2.7}, 
 $$\sum_{i=1}^N \div(A_i)\L_{A_i}B=-\sum_{i=1}^N \div(A_i)\nabla_BA_i.$$
 It follows that $\T_1(fB)=f\T_1(B)$ for any smooth function $f$ on $M$. Concerning equality \eqref{eq2.6}, it was already calculated 
 in proof of Theorem 3.7 of \cite{FangLuo} that 
 \begin{equation*}
 \int_M \Delta (\omega(B))\, dx=- \int_M( \square\omega)(B)\,dx -\int_M \omega(\sum_{i=1}^N
\L_{A_i}^2B)\,dx -2 \int_M\omega(\T_1(B))\, dx.
\end{equation*}
The left hand side of above equality vanishes and 
$\int_M( \square\omega)(B)\,dx=\int_M \omega(\square B)\,dx$,  equality \eqref{eq2.6} follows.
 \end{proof}
 
 \vskip 2mm
 
The tensor $\T_1$ has explicit expression on $\S^n$ (see \cite{FangLuo}). For the sake of self-contained, 
we give here a short presentation. We denote by $\langle\, , \rangle$ the canonical inner product of $\R^{n+1}$. Let $x\in \S^n$, the tangent space $T_x\S^n$ of $\S^n$ at the point $x$ is given by
  \begin{equation*}
  T_x\S^n=\bigl\{v\in\R^{n+1};\ \langle v, x\rangle=0\bigr\}.
  \end{equation*}
Then the orthogonal projection $P_x: \R^{n+1}\ra T_x\S^n$ has the expression:
  \begin{equation*}
  P_x(y)=y-\langle x, y\rangle\, x.
  \end{equation*}

Let ${\mathcal B}=\{e_1, \cdots, e_{n+1}\}$ be an orthonormal basis of $\R^{n+1}$; then the vector field $A_i$ 
defined by $A_i(x)=P_x(e_i)$ has the expression: $A_i(x)=e_i-\langle x, e_i\rangle\, x$ for $i=1, \cdots, n+1$. 
Let $v\in T_x\S^n$ such that $|v|=1$, consider
  \begin{equation*}
  \gamma(t)=x\,\cos t+v\, \sin t.
  \end{equation*}
Then $\{\gamma(t);\ t\in [0,1]\}$ is the geodesic on $\S^n$ such that $\gamma(0)=x, \gamma'(0)=v$. We have $A_i(\gamma(t))= e_i-\langle\gamma(t),e_i\rangle\,\gamma(t)$. Taking the derivative with respect to $t$ and at $t=0$, we get
  \begin{equation}\label{eq2.8}
  (\nabla_vA_i)(x)=P_x\bigl(-\langle v,e_i\rangle x-\langle x, e_i\rangle v\bigr)=-\langle x, e_i\rangle v.
  \end{equation}
It follows that
  \begin{equation}\label{eq2.9}
  \div(A_i)=-n\langle x, e_i\rangle.
  \end{equation}
Hence in this case it is obvious that
  \begin{equation}\label{eq2.10}
  \sum_{i=1}^{n+1} \div(A_i)A_i=-n\sum_{i=1}^{n+1} \bigl(\langle x,e_i\rangle e_i-\langle x, e_i\rangle^2x \bigr) = -n (x-x)= 0.
  \end{equation}
Replacing $v$ by $A_i$ in \eqref{eq2.8}, we have $\dis \nabla_{A_i}A_i=-\langle x, e_i\rangle e_i +\langle x,e_i\rangle^2 x$; therefore summing over $i$, we get
  \begin{equation}\label{eq2.11}
  \sum_{i=1}^{n+1}\nabla_{A_i}A_i=0.
  \end{equation}

Now let $v\in T_x\S^n$ and $a,b\in T_x\S^n$, we have
  \begin{align*}
  \langle A_i\wedge \nabla_vA_i, a\wedge b\rangle&=\langle A_i,a\rangle\langle\nabla_vA_i,b\rangle-\langle A_i,b\rangle\langle\nabla_vA_i,a\rangle\\
  &=-\langle a,e_i\rangle\langle x,e_i\rangle\langle v,b\rangle+\langle x,e_i\rangle\langle b,e_i\rangle\langle v,a\rangle.
  \end{align*}
Summing over $i$ yields
  \begin{equation}\label{eq2.12}
  \sum_{i=1}^{n+1}\langle A_i\wedge \nabla_vA_i, a\wedge b\rangle=\langle a,x\rangle\langle v,b\rangle-\langle x,b\rangle\langle v,a\rangle=0.
  \end{equation}
  
 By \eqref{eq2.10} and \eqref{eq2.11}, we see that the family of vector fields $\{A_1, \ldots, A_{n+1}\}$ 
 satisfy above conditions \eqref{eq2.1}--\eqref{eq2.3}.
Let $B$ be a vector field on $\S^n$; by \eqref{eq2.8}, $\nabla_B A_i=-\<x, e_i\> B$. Using $\dis \L_{A_i}B=\nabla_{A_i}B-\nabla_B A_i$ and combining with \eqref{eq2.9} and \eqref{eq2.10}, we get that
  \begin{equation}\label{eq2.13}
 \T_1(B)= \sum_{i=1}^m \div(A_i) \L_{A_i}B=- n B.
  \end{equation}

\vskip 2mm
In what follows, we will compute $\dis\sum_{i=1}^{n+1}\L_{A_i}^2B$ directly on $\S^n$. Again by torsion free and using \eqref{eq2.8},
we have
\begin{equation*}
\L_{A_i}B=\nabla_{A_i}B-\nabla_B A_i=\nabla_{A_i}B +\langle x, e_i\rangle B,
\end{equation*}
and $\dis \L_{A_i}^2B=\nabla_{A_i}(\L_{A_i}B)-(\nabla_{\L_{A_i}B}) A_i$. Using \eqref{eq2.8} for last term, taking $\nabla_{A_i}$ on the
two sides of above equality, we get

\begin{equation*}
\L_{A_i}^2B=\nabla_{A_i}\nabla_{A_i} B + \nabla_{A_i}(\langle x, e_i\rangle B) + \langle x, e_i\rangle \L_{A_i}B.
\end{equation*}

But $\dis \nabla_{A_i}(\langle x, e_i\rangle B)=\langle A_i, e_i\rangle B + \langle x,e_i\rangle \nabla_{A_i}B$; therefore
\begin{equation*}
\L_{A_i}^2B=\nabla_{A_i}\nabla_{A_i} B+ \langle A_i, e_i\rangle B+ 2  \langle x,e_i\rangle\nabla_{A_i}B
- \langle x,e_i\rangle\nabla_BA_i.
\end{equation*}

We have, by \eqref{eq2.10}
\begin{equation*}
\sum_{i=1}^{n+1}  \langle x,e_i\rangle \nabla_{A_i}=-\frac{1}{n}\sum_{i=1}^{n+1}  \div(A_i) \nabla_{A_i}B= 0.
\end{equation*}
Besides
 $\dis\sum_{i=1}^{n+1} \langle A_i, e_i\rangle=\sum_{i=1}^{n+1}(1-\langle e_i,x\rangle^2)=n$ and
$\dis\sum_{i=1}^{n+1} \langle x,e_i\rangle\nabla_BA_i
=\sum_{i=1}^{n+1}  \langle x,e_i\rangle^2B=B.$
  
According to \eqref{eq2.11}, 
\begin{equation*}
\Delta B=\sum_{i=1}^{n+1}\nabla_{A_i}\nabla_{A_i}B.
\end{equation*}
Therefore combining above calculations, we get
\begin{equation}\label{eq2.14}
\sum_{i=1}^{n+1}\L_{A_i}^2B=\Delta B + (n+1)B.
\end{equation}

Note that $\Ric(B)=(n-1)B$ on $\S^n$, and by Bochner-Weitzenb\"ock formula $-\square = \Delta -\Ric$, we see that 
\eqref{eq2.14} is compatible with \eqref{eq2.6} together with \eqref{eq2.13}.

\begin{proposition}\label{prop2.2} Let $\dis \hat\Delta B=\sum_{i=1}^{n+1} \L_{A_i}^2B$. Then the following dimension free identity holds
\begin{equation*}
-\int_M \langle\hat\Delta B, B\rangle\, dx=2\int_M |\Def(B)|^2\, dx - 2\int_M|B|^2\,dx,\quad\hbox{\rm for any }\div(B)=0.
\end{equation*}
\end{proposition}
\begin{proof} By \eqref{eq1.3},  $\Delta B=-\hat\square B + \Ric(B)$ if $\div(B)=0$,. Then 
$\dis\hat\Delta B=-\hat\square B + 2B$, due to  \eqref{eq2.14}. The result follows.
\end{proof}

\section{Nash embedding and sum of squares of Lie derivatives}\label{sect3}

In what follows, we will find links between above different objects when the compact Riemannian manifold $M$ is isometrically
embedded in an Euclidian space $\R^N$. Let $J: M\ra \R^N$ be a Nash embedding, that is, for any $x\in M$,
\begin{equation}\label{eq3.1}
|dJ(x)v|_{\R^N}=|v|_{T_xM},\quad v\in T_xM.
\end{equation}
 Denote by $(dJ(x))^*: \R^N\ra T_xM$ the adjoint operator of 
$dJ(x)$: 
\begin{equation*}
\langle (dJ(x))^*\xi, v\rangle_{T_xM}=\langle\xi, dJ(x)v\rangle_{\R^N}, \quad \xi\in\R^N, v\in T_xM.
\end{equation*}

For each $x\in M$, we denote $\dis E_x=dJ(x) (T_xM)$ and $E_x^\perp$ the orthogonal in $\R^N$ to $E_x$. Define
\begin{equation*}
\Lambda_x=dJ(x)(dJ(x))^*.
\end{equation*}

We have
\begin{equation*}
\langle \xi-\Lambda_x\xi, dJ(x)v\rangle=\langle \xi, dJ(x)v\rangle- \langle dJ(x)(dJ(x))^*\xi, dJ(x)v\rangle=0.
\end{equation*}
Therefore $\dis \Lambda_x: \R^N\ra E_x$ is the orthogonal projection. Set $\dis \Lambda_x^\perp={\rm Id}-\Lambda_x$.
 By polarization of \eqref{eq3.1}, we have
 \begin{equation*}
 (dJ(x))^*dJ(x)={\rm Id}_{T_xM}.
 \end{equation*}

The mapping $\Lambda: M\ra L(\R^N, \R^N)$ from $M$ to the space of linear maps of $\R^N$ is smooth. For $\xi\in \R^N$ given, $v\in T_xM$, 
consider a smooth curve $\gamma$ on $M$ such that $\gamma(0)=x, \gamma'(0)=v$, we denote

\begin{equation*}
d\Lambda_x(v) \xi=\frac{d}{dt}_{|_{t=0}}\Bigl(\Lambda_{\gamma(t)}\xi\Bigr)\in\R^N.
\end{equation*}

\begin{proposition}\label{prop3.1}(see \cite{Stroock},ch.5) Let $x\in M$, for any $v\in T_xM$, it holds true:
\begin{equation}\label{eq3.2}
(i)\quad \Lambda_x\circ  d\Lambda_x=d\Lambda_x\circ \Lambda_x^\perp,\hskip 15mm
(ii)\quad \Lambda^\perp_x\circ  d\Lambda_x=d\Lambda_x\circ \Lambda_x.
\end{equation}
\end{proposition}

\begin{proof}
We have $\dis \Lambda_x(\Lambda_x\xi)=\Lambda_x\xi$.  Taking the derivative with respect to $x$ in direction $v$ 
on two sides of this equality, we get
\begin{equation*}
d\Lambda_x(v)(\Lambda_x\xi)+\Lambda_x(d\Lambda_x(v)\xi)=d\Lambda_x(v)\xi,
\end{equation*}
which yields the first equality in \eqref{eq3.2}. For $(ii)$, it suffices to use 
$\dis \Lambda_x^\perp(\Lambda_x^\perp\xi)=\Lambda_x^\perp \xi$ and
$d\Lambda_x^\perp(v)=-d\Lambda_x(v)$.

\end{proof}

Now let $\xi\in E_x$, then by $(i)$ in \eqref{eq3.2}, $\dis \Lambda_x(d\Lambda_x(v)\xi)=0$, which implies that
$d\Lambda_x(v)\xi\in E_x^\perp$. This means that $d\Lambda_x(v)$ sends $E_x$ into $E_x^\perp$. By $(ii)$ in \eqref{eq3.2},
$d\Lambda_x(v)$ sends $E_x^\perp$ into $E_x$.

\vskip 2mm
Now we introduce the Levi-Civita covariant derivative on $M$: for $B\in\X(M)$ and $v\in T_xM$, set
\begin{equation}\label{eq3.3}
(\nabla_vB)(x)=(dJ(x))^*\Bigl(\partial_v(dJ(\cdot)B_\cdot)(x)\Bigr),
\end{equation}
where $\partial_v$ denotes the derivative on the manifold $M$ in direction of $v$. If $B_x=(dJ(x))^*\xi$ for a fix $\xi$, then
\begin{equation*}
\langle \nabla_vB, u\rangle_{T_xM}=\langle d\Lambda_x(v)\xi, dJ(x)u\rangle_{\R^N},\quad u\in T_xM.
\end{equation*}

\vskip 2mm
The {\it second fundamental form} $\alpha$ on $M$ is defined as follows: for $u,v\in T_xM$,
\begin{equation}\label{eq3.4}
\alpha_x(u,v)=d\Lambda_x(u)\cdot dJ(x)v.
\end{equation}

By Proposition \ref{prop3.1}, $\alpha_x(u,v)\in E_x^\perp$. For $X, Y\in\X(M)$, we define $\alpha(X,Y)(x)=\alpha_x(X_x,Y_x)$.
 It is clear that 
$\alpha(X, fY)=f\,\alpha(X,Y)$. In what follows, we will see that $\alpha$ is symmetric bilinear application: 
$\alpha(X,Y)=\alpha(Y,X)$.

\vskip 2mm
Let $X$ be a vector field on $M$, then $x\ra dJ(x)X_x$ is a $\R^N$-valued function; therefore there is a function $\bar X: \R^N\ra\R^N$ such that $\bar X(J(x))=dJ(x)X_x$ for $x\in M$. Let $Y$ be another vector field on $M$, and $\bar Y: \R^N\ra \R^N$ such that
$\bar Y(J(x))=dJ(x)Y_x, x\in M$. Taking the derivative of $x\ra \bar Y(J(x))$ along $X$:

\begin{equation}\label{eq3.5}
\partial_X\bigl(\bar Y\circ J\bigr)(x)=\bar Y'(J(x))\cdot dJ(x)X_x=\bar Y'(J(x))\bar X(J(x))=(D_{\bar X}\bar Y)(J(x)),
\end{equation}
where $\bar Y'$ denotes the differential of $\bar Y$ and $D_{\bar X}$ the derivative on $\R^N$ with respect to $\bar X$. It follows that
\begin{equation}\label{eq3.6}
[\bar X, \bar Y](J)=\partial_X(\bar Y(J))-\partial_Y(\bar X(J)).
\end{equation}

Since $\bar Y(J)(x)\in E_x$, then $\Lambda_x(\bar Y(J(x)))=\bar Y(J(x))$. Taking the derivative with respect to $X$ and according to 
definition \eqref{eq3.4}, we get

\begin{equation*}
\begin{split}
\partial_X(\bar Y(J))(x)&=d\Lambda_x(X_x)\cdot \bar Y(J(x))+ \Lambda_x\Bigl(\partial_X(\bar Y(J))(x)\Bigr)\\
&=\alpha_x(X_x,Y_x)+ \Lambda_x\Bigl(\partial_X(\bar Y(J))(x)\Bigr).
\end{split}
\end{equation*}

Combining this with \eqref{eq3.6}, we get
\begin{equation*}
[\bar X, \bar Y](J)=\alpha(Y,X)-\alpha(X,Y)+ \Lambda_\cdot \bigl( [\bar X, \bar Y](J)\bigr).
\end{equation*}

By property of embedding, $[\bar X, \bar Y](J(x))\in E_x$ so that $\Lambda_\cdot \bigl( [\bar X, \bar Y](J)\bigr)=[\bar X, \bar Y]$. Then
the symmetry of $\alpha$ follows.

\vskip 2mm
Now according to definition \eqref{eq3.3} and to \eqref{eq3.5}, we get $\dis (\nabla_XY)(x)=(dJ(x))^*(D_{\bar X}\bar Y)(J(x))$. Therefore 
the following orthogonal decomposition holds

\begin{equation}\label{eq3.7}
(D_{\bar X}\bar Y)(J(x))=dJ(x) (\nabla_XY)(x)+\alpha_x(X_x,Y_x).
\end{equation}
\vskip 2mm

Now let $x\ra V_x\in E_x^\perp$ be a  field of normal vectors to $M$. Set, for $X\in\X(M)$,
\begin{equation}\label{eq3.8}
\A(X,V)(x)=-(dJ(x))^*(\partial_X V)(x).
\end{equation}

Since $x\ra \langle dJ(x)Y_x, V_x\rangle_{\R^N}=0$, taking the derivative with respect to $X$ yields
\begin{equation*}
\langle \partial_X (dJ(\cdot)Y)(x), V_x\rangle_{\R^N} + \langle dJ(x)Y_x, (\partial_XV)(x)\rangle_{\R^N}=0.
\end{equation*}
 
 By \eqref{eq3.7} and \eqref{eq3.8}, we get
 \begin{equation}\label{eq3.9}
\langle \alpha(X,Y), V\rangle= \langle \A(X,V), Y\rangle .
\end{equation}
 
From above expression, $\A(X, fV)=f\A(X,V)$. $\A$ is called {\it shape operato}r of $M$. 

\begin{proposition}\label{prop3.2} Let $B\in\X(M)$ defined by $B_x=(dJ(x))^*\xi$ for a fixed $\xi\in\R^N$. Then 
\begin{equation}\label{eq3.10}
\nabla_XB=\A(X, \Lambda_\cdot^\perp\xi),\quad X\in\X(M).
\end{equation}
\end{proposition}

\begin{proof} By \eqref{eq3.3}, $\dis (\nabla_vB)(x)=(dJ(x))^*\bigl(\partial_v (\Lambda_\cdot \xi)\bigr)$ which is equal to
$-(dJ(x))^*\bigl(\partial_v (\Lambda_\cdot^\perp \xi)\bigr)$ as $\Lambda_x\xi+\Lambda_x^\perp=\xi$. Now definition \eqref{eq3.8} gives
the result.
\end{proof}
\vskip 2mm

For the sake of self-contained, we use \eqref{eq3.10} to check properties \eqref{eq2.2}, \eqref{eq2.3} and  \eqref{eq2.7}. 
Let ${\mathcal B}=\{e_1, \ldots, e_N\}$ be an orthonormal basis of $\R^N$, we define
\begin{equation*}
A_i(x)=(dJ(x))^*e_i,\quad i=1, \ldots, N.
\end{equation*}

\begin{proposition}\label{prop3.3} We have
\begin{equation*}
\sum_{i=1}^N \nabla_{A_i}A_i =0.
\end{equation*}
\end{proposition}

\begin{proof} Let $u\in T_xM$, then by \eqref{eq3.10} and \eqref{eq3.9} respectively, 
\begin{equation*}
\langle (\nabla_{A_i}A_i)(x), u\rangle=\langle \A(A_i, \Lambda_\cdot^\perp\xi_i), u\rangle
=\langle \alpha_x((dJ(x))^*e_i, u), e_i\rangle.
\end{equation*}
The sum from $i=1$ to $N$ of above terms is basis independent. Then taking $\{e_1, \ldots, e_n\}\subset E_x$
and $\{e_{n+1}, \ldots, e_N\}\subset E_x^\perp$, and remarking $(dJ(x))^*e_j=0$ for $j>n$, we have

\begin{equation*}
\langle \sum_{i=1}^N(\nabla_{A_i}A_i)(x), u\rangle=\sum_{i=1}^n \langle \alpha_x((dJ(x))^*e_i, u), e_i\rangle=0,
\end{equation*}
due to the orthogonality. The result follows.
\end{proof}

\begin{proposition}\label{prop3.4} We have
\begin{equation*}
\sum_{i=1}^N  A_i\wedge\nabla_V A_i=0,\quad V\in\X(M).
\end{equation*}
\end{proposition}
\begin{proof}
Let $a,b\in T_xM$, we have
\begin{equation*}
\langle A_i\wedge\nabla_VA_i, a\wedge b\rangle=\langle A_i, a\rangle \langle\nabla_V A_i, b\rangle
-\langle A_i, b\rangle\langle \nabla_VA_i, a\rangle,
\end{equation*}
which is equal to, by \eqref{eq3.10},
\begin{equation*}
\langle A_i, a\rangle \langle\A(V_x, \Lambda_x^\perp e_i), b\rangle
-\langle A_i, b\rangle \langle\A(V_x, \Lambda_x^\perp e_i), a\rangle,
\end{equation*}

which is equal to, by \eqref{eq3.9},
\begin{equation*}
\langle e_i, dJ(x)a\rangle \langle\alpha(V_x, b), e_i\rangle
-\langle e_i, dJ(x)b\rangle \langle\alpha(V_x, a), e_i\rangle.
\end{equation*}

Therefore 
\begin{equation*}
\sum_{i=1}^N \langle A_i\wedge\nabla_VA_i, a\wedge b\rangle
= \langle \alpha(V_x,b), dJ(x)a\rangle- \langle \alpha(V_x,a), dJ(x)b\rangle=0.
\end{equation*}
\end{proof}

\begin{proposition}\label{prop3.5} Let $B(x)=(dJ(x))^*\xi$ for a fixed $\xi\in\R^N$. Then
\begin{equation}\label{eq3.11}
\div(B)=\langle {\rm Trace}(\alpha), \xi\rangle_{\R^N}.
\end{equation}
\end{proposition}

\begin{proof} Let $\{v_1, \ldots, v_n\}$ be an orthonormal basis of $T_xM$; by \eqref{eq3.10} and \eqref{eq3.9},
\begin{equation*}
\div(B)=\sum_{i=1}^n \langle \A(v_i, \Lambda_x^\perp\xi), v_i\rangle=\sum_{i=1}^n \langle \alpha(v_i, v_i), \xi\rangle
=\langle {\rm Trace}(\alpha), \xi\rangle_{\R^N}.
\end{equation*}
\end{proof}

\begin{theorem}\label{th3.1} Let ${\mathcal B}=\{e_1, \ldots, e_N\}$ be an orthonormal basis of $\R^N$, and
$A_i(x)=(dJ(x))^*e_i$. Define
\begin{equation}\label{eq3.12}
\T_1(B)=\sum_{i=1}^N \div(A_i)\L_{A_i}B.
\end{equation}
Then 
\begin{equation}\label{eq3.13}
\T_1(B)=-\A(B, {\rm Trace}(\alpha)).
\end{equation}
\end{theorem}

\begin{proof} We first see, by \eqref{eq3.11}, that
\begin{equation*}
\sum_{i=1}^N \div(A_i)(x)A_i(x)=\sum_{i=1}^N \langle {\rm Trace}(\alpha_x), e_i\rangle(dJ(x))^*e_i
=(dJ(x))^*({\rm Trace}(\alpha_x))=0,
\end{equation*}
since ${\rm Trace}(\alpha_x)\in E_x^\perp$. Therefore term \eqref{eq3.12} becomes 
$\dis \T_1(B)=\sum_{i=1}^N \div(A_i)\nabla_BA_i$. Again using \eqref{eq3.10}, 
\begin{equation*}
\T_1(B) = -\sum_{i=1}^N \langle {\rm Trace}(\alpha), e_i\rangle\A(B, \Lambda_\cdot^\perp e_i)=-\A(B, {\rm Trace}(\alpha)).
\end{equation*}

\end{proof}
{\it Remark:} In the case of $\S^n$, for $x\in \S^n$, the outer normal unit vector ${\bf n}(x)=x$; then for $v\in T_xM$ and 
$\gamma(t)=x\cos(t)+v\sin(t)$, we see that $\dis \frac{d}{dt}_{|_{t=0}}{\bf n}(\gamma(t))=v\in T_xM$. Therefore 
\begin{equation*}
\A(v,{\bf n})=v.
\end{equation*}
Using \eqref{eq3.13} gives $\dis \T_1(B)= -nB$.

\vskip 2mm
In what follows, we will give a link between the Ricci tensor and $\T_1$. Let $R(B,Y)Z=[\nabla_B, \nabla_Y]Z-\nabla_{[B,Y]}Z$ be the
curvature tensor on $M$. Using second fundamental form, $R$ can be expressed by (see \cite{KN}):
\begin{equation}\label{eq3.14}
\langle R(B,Y)Z, W\rangle=\langle \alpha(Y,Z), \alpha(B, W)\rangle - \langle \alpha(B,Z), \alpha(Y, W)\rangle.
\end{equation}

Let $A_i$ be defined as above. Define
\begin{equation*}
\psi(B,W)=\sum_{i=1}^N \langle \alpha(B,A_i), \alpha(W,A_i)\rangle,
\end{equation*}
which is basis independent, symmetric bilinear form. Using \eqref{eq3.9}, we get
\begin{equation*}
\psi(B,W)=\sum_{i=1}^N \langle \A(B,\Lambda_\cdot^\perp e_i), \A(W,\Lambda_\cdot^\perp e_i)\rangle.
\end{equation*}

This quantity is independent of the dimension of $M$, but the co-dimension of $M$ in $\R^N$.  Define the tensor $\T_2$ by
\begin{equation*}
\langle \T_2(B), W\rangle=\psi(B, W),\quad W\in\X(M).
\end{equation*}
The tensor $\T_2$ is directly related to the manner of the embedding $M$ into $\R^N$.

\begin{proposition}\label{prop3.7} It holds true
\begin{equation}\label{eq3.15}
\T_1=- \Ric -\T_2.
\end{equation}
\end{proposition}

\begin{proof} We have $\dis \langle \Ric(B), W\rangle=\sum_{i=1}^N\langle R(B, A_i)A_i, W\rangle$. The relation \eqref{eq3.14} 
leads $\dis \langle \Ric(B), W\rangle$ to
\begin{equation*}
\begin{split}
&\sum_{i=1}^N \langle \alpha(A_i, A_i), \alpha(B,W)\rangle-\sum_{i=1}^N \langle \alpha(B, A_i), \alpha(W,A_i)\rangle\\
&=\langle \A(B, {\rm Trace}(\alpha)), W\rangle-\psi(B,W)=-\langle \T_1(B), W\rangle-\langle \T_2(B), W\rangle.
\end{split}
\end{equation*}
\end{proof}
The result \eqref{eq3.15} follows.

\begin{theorem}\label{th3.8} Let $B$ be a vector field on $M$ of divergence free, then
\begin{equation}\label{eq3.16}
\sum_{i=1}^N \L_{A_i}^2B=-\hat\square B  + 2 \T_2(B).
\end{equation}
\end{theorem}

\begin{proof} Equality \eqref{eq3.16} follows from \eqref{eq2.6} and \eqref{eq3.15}.
\end{proof}

\section{Probabilistic representation formula for the Navier-Stokes equation}

Let's first state Constantin-Iyer's probabilistic representation formula for Navier-Stokes equation \eqref{NSE} on ${\mathbb T}^n$.

\begin{theorem}[Constantin--Iyer]\label{thCI}
Let $\nu>0$, $W$ be an $n$-dimensional Wiener process, $k\geq 1$, and $u_0\in C^{k+1,\alpha}$ a given deterministic divergence-free vector field.
Let the pair $(X,u)$ satisfy the stochastic system
  \begin{equation}\label{eq4.1}
  \begin{cases}
  \d X_t=\sqrt{2\nu}\,\d W_t+u_t(X_t)\,\d t,\\
  u_t=\E\mathbf{P}\big[\big(\nabla X_t^{-1}\big)^\ast \big(u_0\circ X_t^{-1}\big)\big],
  \end{cases}
  \end{equation}
where $\mathbf{P}$ is the Leray--Hodge projection and the star $\ast$ denotes the transposed matrix.
Then $u$ satisfies the incompressible Navier--Stokes equations \eqref{NSE}.
\end{theorem}

For a given sufficiently regular initial velocity $u_0$, there is a $T>0$ such that system \eqref{eq4.1} admits a unique solution over 
$[0,T]$ and $X_t$ is a diffeomorphism of ${\mathbb T}^n$ (see \cite{Constantin}). Using the terminology of pull-back by diffeomorphism of vector fields, the following intrinsic formulation to the second identity in \eqref{eq4.1}  was given in \cite{FangLuo}:

  \begin{equation*}
  \int_{{\mathbb T^n}}\<u_t,v\>\,\d x=\E\bigg(\int_{{\mathbb T^n}}\big\<u_0,(X_{t}^{-1})_*v\big\>\,\d x\bigg),\quad \div(v)=0,\ \forall\, t\geq0,
  \end{equation*}
which means that the evolution of $u_t$ in the direction $v$ is equal to the average of the evolution of $v$ under the inverse flow $X_t^{-1}$ in the initial direction $u_0$. The generalization of Theorem \ref{thCI} to a Riemannian manifold gave rise to a problem how to decompose the De Rham-Hodge Laplacian on vector fields as sum of squares of Lie derivatives : this has been achieved in \cite{FangLuo} on Riemannian symmetric spaces.

\vskip 2mm
The purpose of this section is to derive a Feymann-Kac type representation for solutions to the following Navier-Stokes equation on $M$:

\begin{equation}\label{eq4.2}
  \begin{cases}
  \partial_t u+\nabla_u u+\nu\square u+\nabla p=0,\\
  \div(u)=0,\quad u|_{t=0}=u_0
  \end{cases}
  \end{equation}
where $\nu>0$.
\vskip 2mm

Consider a Nash embedding $J: M\ra \R^N$ as in Section \ref{sect3}; let ${\mathcal B}=\{e_1, \ldots, e_N\}$ be an orthonormal 
basis of $\R^N$ and $A_i(x)= (dJ(x))^*e_i, x\in M$. Let $\{u_t,\ t\in [0,T]\}$ be a time-dependent vector fields  in $C^{1,\beta}$
with $\beta>0$, then the following Stratanovich SDE on $M$

  \begin{equation}\label{eq4.3}
  d X_t=\sqrt{2\nu}\,\sum_{i=1}^N A_i(X_t)\circ dW_t^i + u_t(X_t)\,dt,\quad X_0=x\in M
  \end{equation}
 admits a unique solution $\{X_t, t\in [0,T]\}$ which defines a flow of diffeomorphisms of $M$ (see \cite{Bismut, IW, Elworthy}), 
 where $\{W_t^i; \, i=1, \ldots, N\}$ is a $N$-dimensional standard Brownian motion.
 
 \vskip 2mm
 The space $\X(M)$ of vector fields on $M$,  equipped with uniform norm 
 $\dis ||B||_\infty=\sup_{x\in M}|B_x|_{T_xM}$, is a Banach space.
  We equip $L(\X, \X)$, the space of linear map from $\X(M)$ into $\X(M)$, the norm
 of operator: 
 \begin{equation*}
 |||Q|||=\sup_{||B||_\infty\leq 1}||QB||_\infty.
 \end{equation*}
 
 For a diffeomorphism $\varphi: M\ra M$,  the pull-back $\varphi_*$  sends $\X(M)$ into $\X(M)$,  
 which is in the space $L(\X,\X)$ with $|||\varphi_*|||\leq ||\varphi||_{C^1}$. For a tensor $\T$ which sends $\X(M)$ to $\X(M)$, we denote
 $(\T_*B)(x)=\T(x)B_x$, then $|||\T_*|||\leq ||T||_\infty$.
 Now we consider the following linear differential equation on $L(\X,\X)$:
 \begin{equation}\label{eq4.4}
 \frac{dQ_t}{dt}= \nu\, Q_t\cdot (X_t^{-1})_*\ \T_*\ (X_t)_*,\quad Q_0={\rm identity}
 \end{equation}
 where $(X_t)$ is solution to SDE \eqref{eq4.3}.
 
 \begin{theorem}\label{th4.2} Let $\dis -\tilde\square = \sum_{i=1}^N \L_{A_i}^2 +\T_*$. Then  if $\tilde\square$ preserves the space of
 vector fields of divergence free, $u_t\in C^{2,\beta}$ is solution to 
 \begin{equation}\label{eq4.5}
  \begin{cases}
  \partial_t u_t+\nabla_{u_t} u_t+\nu\tilde\square u_t+\nabla p=0,\\
  \div(u_t)=0,\quad u|_{t=0}=u_0
  \end{cases}
  \end{equation}
 if and only if, for each $v\in\X(M)$ with $\div(v)=0$,
 
  \begin{equation}\label{eq4.6}
  \int_M \<u_t,v\>\, dx=\E\bigg(\int_M \big\<u_0,Q_t\,(X_{t}^{-1})_*v\big\>\, dx\bigg),\  \forall\, t\in [0,T].
  \end{equation}

Moreover, $u_t$ has the following more geometric expression
  \begin{equation}\label{eq4.7}
  u_t=\E\Bigl[ {\mathbf P}\Bigl(\rho_t\cdot (X_t^{-1})^\ast (\tilde Q_tu_0)^\flat\Bigr)^{\#}\Bigr]
  \end{equation}
where $\tilde Q_t$ is the adjoint of $Q_t$ in the sense that 
$\int_M \langle \tilde Q_t\, u, v\rangle\,dx=\int_M \langle u, Q_t\, v\rangle\,dx$, $\rho_t$ is the density of $(X_t)_\#(dx)$ with 
respect to $dx$ and ${\mathbf P}$ is the Leray-Hodge projection on the space of vector fields of divergence free.

\end{theorem}
 
 \begin{proof} Suppose \eqref{eq4.6} holds. By It\^o formula \cite[p.265, Theorem 2.1]{Kunita}, we have
  \begin{equation*}
  d (X^{-1}_t)_*(v)
  =\sum_{i=1}^N (X^{-1}_t)_*(\L_{A_i} v)\, dW^i_t+\nu \sum_{i=1}^N  (X^{-1}_t)_*(\L_{A_i}^2 v)\,dt
  + (X^{-1}_t)_*(\L_{u_t}v) \,dt.
\end{equation*}
 
We have, according to \eqref{eq4.4} and above formula,
\begin{equation*}
\begin{split}
d \Bigl[ Q_t\cdot (X^{-1}_t)_*v\Bigr]&=dQ_t\cdot (X^{-1}_t)_*v + Q_t\cdot d(X^{-1}_t)_*v\\
&=\nu\, Q_t\cdot (X^{-1}_t)_*\T_*v\, dt +\sum_{i=1}^N Q_t\cdot (X^{-1}_t)_*(\L_{A_i} v)\, dW^i_t\\
&\hskip 5mm +\nu \sum_{i=1}^N  Q_t\cdot (X^{-1}_t)_*(\L_{A_i}^2 v)\,dt
  + Q_t\cdot (X^{-1}_t)_*(\L_{u_t}v) \,dt,
 \end{split}
\end{equation*} 
since $(X_t)_*(X^{-1}_t)_*={\rm identity}$. Now using definition of $\tilde\square$, we get

  \begin{equation*}
  \begin{split}
  \int_M  \langle u_t, v\rangle\,dx=&\int_M \langle u_0, v\rangle\, dx 
  -\nu\,  \int_0^t \E\bigg(\int_M \big\langle u_0, Q_s\cdot(X^{-1}_s)_*(\tilde\square v)\big\rangle\, dx\bigg) ds\\
  &+\int_0^t \E\bigg(\int_M \big\langle u_0,\, Q_s\cdot(X^{-1}_s)_*(\L_{u_s}v)\big\rangle\, dx\bigg) ds.
  \end{split}
  \end{equation*}

Since $\tilde\square v$ and $\L_{u_s}v$ are of divergence free, to them we apply \eqref{eq4.6} to get

\begin{equation*}
\E\bigg(\int_M \big\langle u_0, Q_s\cdot(X^{-1}_s)_*(\tilde\square v)\big\rangle\, dx\bigg)
=\int_M \<u_s,\tilde\square v\>\, dx,
\end{equation*}

and 
\begin{equation*}
\E\bigg(\int_M \big\langle u_0, Q_s\cdot(X^{-1}_s)_*(\L_{u_s} v)\big\rangle\, dx\biggr)
=\int_M \<u_s,\L_{u_s} v\>\, dx,
\end{equation*}

Therefore we obtain the following equality:
\begin{equation*}
  \int_M\<u_t,v\>\,dx=\int_M\<u_0,v\>\,dx-\nu \int_0^t\!\!\int_M \<u_s, \tilde\square v \>\,dxds
  +\int_0^t\!\!\int_M \<u_s,\L_{u_s} v \>\,dxds.
  \end{equation*}
  
But the last term can be changed 
 \begin{align*}
  \int _M \<u_s,\L_{u_s} v \>\,dx&=\int_M \<u_s,\nabla_{u_s} v \>\,dx-\int_M \<u_s,\nabla_v u_s \>\,dx\\
  &=\int_M \<u_s,\nabla_{u_s} v \>\,dx-\frac12 \int_M v(|u_s|^2)\,dx
  =\int_M \<u_s,\nabla_{u_s} v \>\,dx.
  \end{align*}
Finally 
\begin{equation*}
  \int_M\<u_t,v\>\,dx=\int_M\<u_0,v\>\,dx-\nu \int_0^t\!\!\int_M \<u_s,\tilde\square v \>\,dxds
  +\int_0^t\!\!\int_M \<u_s,\nabla_{u_s} v \>\,dxds.
  \end{equation*}
  
It follows that for a.e. $t\geq 0$, 

\begin{equation*}
 \frac{\d}{\d t}\int_M \<u_t,v\>\,dx=\nu \int_M \<u_t,\tilde\square v \>\,dx  +\int_M \<u_t,\nabla_{u_t} v \>\,dx.
 \end{equation*}
 
Multiplying both sides by $\gamma\in C_c^1([0,\infty))$ and integrating by parts on $[0,\infty)$, we get

  \begin{equation*}
  \gamma(0)\int_M\<u_0,v\>\,dx+\int_0^\infty\!\!\int_M \big[\gamma'(t)\<u_t,v\> +\gamma(t) \<u_t,\nabla_{u_t} v \>
  +\nu\, \gamma(t)\<u_t,\tilde\square v\> \big] dxdt=0.
  \end{equation*}
The above equation is the weak formulation of the Navier--Stokes \eqref{eq4.5} on the manifold $M$. 

\vskip 2mm
To prove \eqref{eq4.7}, we note that
  \begin{align*}
  \int_M \rho_t\, \big\< (X_t^{-1})^\ast (\tilde Q_tu_0)^\flat, v\big\>\,dx
  &=\int_M \rho_t\, \big\< (\tilde Q_tu_0)^\flat, (X_t^{-1})_\ast v\big\>_{X_t^{-1}}\,dx\\
  &=\int_M \rho_t(X_t)\,\tilde\rho_t\, \big\< u_0^\flat,\ Q_t\cdot(X_t^{-1})_\ast v\big\> \,dx\\
  &=\int_M \big\< u_0^\flat, \ Q_t\cdot(X_t^{-1})_\ast v\big\> \,dx=\int_M \big\< u_0,\ Q_t\cdot (X_t^{-1})_\ast v \big\>_{T_x M}\,dx
  \end{align*}
where $\tilde\rho_t$ is the density of  $(X_t^{-1})_\#(dx)$ with respect to $dx$.
Now by \eqref{eq4.6}, we have:
\begin{equation*}
   \int_M \langle u_t, v\rangle\,dx=\E\bigg(\int_M \rho_t\, \big\< (X_t^{-1})^\ast (\tilde Q_tu_0)^\flat, v\big\>\,dx\bigg), \quad \div(v)=0.
 \end{equation*}

The formula \eqref{eq4.7} follows.

\vskip 2mm
For proving the converse, we use the idea in \cite[Theorem 2.3]{Zhang}. Let $u_t\in C^2 $ be a solution to \eqref{eq4.5}, then
  \begin{equation*}
  \int_M\<u_t,v\>\ dx=\int_M\<u_0,v\>\, dx-\nu \int_0^t\!\!\int_M \<u_s, \tilde\square v\>\, dxds
  +\int_0^t\!\!\int_M \<u_s,\L_{u_s} v\>\, dxds.
  \end{equation*}
  
 Define
  \begin{equation*}
  \tilde u_t=\E\Bigl[ {\mathbf P}\Bigl(\rho_t\cdot (X_t^{-1})^\ast (\tilde Q_tu_0)^\flat\Bigr)^{\#}\Bigr].
  \end{equation*}
Then  calculations as above lead to
  \begin{equation*}
  \int_M\<\tilde u_t,v\>\,\d x=\int_M\<u_0,v\>\,\d x-\nu \int_0^t\!\!\int_M \<\tilde u_s, \tilde\square v\>\,\d x\d s
  +\int_0^t\!\!\int_M \<\tilde u_s,\L_{u_s} v\>\,\d x\d s.
  \end{equation*}
Let $z_t=u_t-\tilde u_t$; we have
  \begin{equation*}
  \int_M\<z_t,v\>\,\d x= -\nu \int_0^t\!\!\int_M \<z_s, \tilde\square v\>\,\d x\d s +\int_0^t\!\!\int_M \<z_s,\L_{u_s} v\>\,\d x\d s.
  \end{equation*}
It follows that $(z_t)$ solves the following heat equation on $M$
  \begin{equation*}
  \frac{\d z_t}{\d t}=-\nu\,\tilde\square z_t - \L_{u_t}^*z_t,\quad z_0=0,
  \end{equation*}
  where $\L_{u_t}^*$ is the adjoint operator. 
By uniqueness of solutions, we get that $z_t=0$ for all $t\geq 0$. Thus $u_t=\tilde u_t$.
Then \eqref{eq4.7} follows. The proof of Theorem \ref{th4.2} is complete.
\end{proof}

\vskip2mm
\begin{remark}{\rm
By Proposition \ref{prop2.1}, $\dis -\square B= \sum_{i=1}^N \L_{A_i}^2B + 2 \T_1(B)$; Taking $\T=2\T_1$ in Theorem \ref{th4.2}, we obtain
a probabilistic representation formula for Navier-Stokes equation \eqref{eq4.2} on a general compact Riemnnian manifold, since $\square$ preserves 
the space of vector fields of divergence free.   According to Theorem \ref{th3.8}, the Ebin-Marsden Laplacian $\hat\square$
has the expression $\dis -\hat\square B= \sum_{i=1}^N \L_{A_i}^2B - 2 \T_2(B)$. However, by \eqref{eq1.3} and \eqref{eq1.4},

\begin{equation*}
\hat\square B= \square B + 2\Ric (B)\quad \hbox{\rm for } \div(B)=0;
\end{equation*}

therefore $\hat\square$ {\it does not} preserve
the space of vector fields of divergence free, except for the case where $\Ric=k\, {\rm Id}$, that is to say that $M$ is a Einstein manifold.}
\end{remark}

\begin{remark} {\rm Let $g$ be the Riemannian metric of $M$. A vector field $A$ on $M$ is said to be a Killing vector field if
\begin{equation*}
\L_A g=0.
\end{equation*}
Let $X, Y\in\X(M)$, we have, using Lie derivatives, 
\begin{equation*}
\partial_A \bigl( g(X,Y)\bigr)=(\L_A g)(X,Y)+ g(\L_AX,Y)+g(X,\L_AY), \eqno(i)
\end{equation*}
and using covariant derivatives,

\begin{equation*}
\partial_A \bigl( g(X,Y)\bigr)= g(\nabla_AX,Y)+g(X,\nabla_AY).\eqno(ii)
\end{equation*}
Since $\nabla_AX-\L_AX=\nabla_XA$,  combining $(i)$ and $(ii)$ yields
\begin{equation*}
(\L_Ag)(X,Y)= g(\nabla_XA,Y)+g(X,\nabla_YA)=2\Def(A)(X,Y).
\end{equation*}

That is to say that $A$ is a Killing vector field if and only it $\Def(A)=0$, which implies $\hat\square A=0$. Conversely if $\div(A)=0$
and $\hat\square A=0$, then $A$ is a Killing vector field on $M$ (see \cite{Yano}).  Does is it why Ebin and Marsden said that 
$\hat\square$ is more convenient in \cite{EM} ?
}
\end{remark}

\begin{proposition} Let $\rho_t$ be given in Theorem \ref{th4.2}; then the following stochastic partial equation (SPDE) holds
\begin{equation}\label{eq4.8}
d\rho_t = -\sqrt{2\nu}\bigl\langle dJ(x)\nabla\rho_t +{\rm Trace}(\alpha)\, \rho_t,\ dW_t\bigr\rangle
+\bigl( \nu\Delta \rho_t + \L_{u_t}\rho_t\bigr)\, dt.
\end{equation}

\end{proposition}
\begin{proof} Let $f$ be a $C^2$-function on $M$, by It\^o formula, we have
\begin{equation*}
d\, f(X_t)=\sqrt{2\nu}\sum_{i=1}^N (\L_{A_i}f)(X_t)\, dW_t^i + \nu\sum_{i=1}^N(\L_{A_i}^2 f)(X_t)\, dt + (\L_{u_t}f)(X_t)\, dt.
\end{equation*}
Since $\dis\Delta f=\sum_{i=1}^N\L_{A_i}^2f$ and 
\begin{equation*}
\int_M( \L_{A_i}f) (X_t)\ dx=\int_M \L_{A_i}f\, \rho_t\ dx=-\int_M f\ \bigl(\div(A_i)\rho_t + \L_{A_i}\rho_t\bigr)\, dx,
\end{equation*}
\begin{equation*}
\int_M (\Delta f)(X_t)\, dx=\int_M \Delta f\, \rho_t\ dx=\int_M f\, \Delta\rho_t\ dx, 
\end{equation*}

\begin{equation*}
 \int_M \L_{u_t}(X_t)\ dx=\int_M \L_{u_t}f \ \rho_t\ dx= \int_M f\ \L_{u_t}\rho_t\, dx,
\end{equation*}

and according to $\int_Mf(X_t)\ dx=\int_M f\ \rho_t\, dx$,
we get that $\rho_t$ satisfies the following SPDE
 \begin{equation}\label{eq4.9}
 d\rho_t = - \sqrt{2\nu} \sum_{i=1}^N \Bigl( \div(A_i)\rho_t + \L_{A_i}\rho_t\Bigr)\, dW_t^i 
 + \Bigl( \nu\Delta\rho_t + \L_{u_t}\rho_t\Bigr)\, dt.
\end{equation}

Using \eqref{eq3.11}, we have
\begin{equation*}
\sum_{i=1}^N \div(A_i)\rho_t\,dW_t^i= \langle{\rm Trace}(\alpha)\rho_t, \ dW_t\rangle.
\end{equation*}

Besides,
$\dis \sum_{i=1}^N \langle (dJ(x))^*e_i, \nabla\rho_t\rangle \, dW_t^i
=\sum_{i=1}^N \langle e_i, dJ(x)\nabla\rho_t\rangle \, dW_t^i=\langle dJ(x)\nabla\rho_t, dW_t\rangle$.
Therefore \eqref{eq4.8} follows from \eqref{eq4.9}.

\end{proof}


\begin{thebibliography}{99}


\bibitem{AC1} M. Arnaudon, A.B. Cruzeiro, Lagrangian Navier-Stokes diffusions on manifolds: variational principle and stability, 
\emph{Bull. Sci. Math.}, 136 (8) (2012),  857--881.

\bibitem{AC2} M. Arnaudon, A.B. Cruzeiro, Stochastic Lagrangian flows on some compact manifolds, \emph{ Stochastics}, 84  (2012), 367--381.

\bibitem{ACF}  M. Arnaudon, A.B. Cruzeiro, S. Fang, Generalized stochastic Lagrangian paths for the Navier--Stokes equation, 
\emph{Ann. Sc. Norm. Super. Pisa}, CI. Sci., 18 (2018), 1033--1060.


\bibitem{ACLZ} M Arnaudon, A B. Cruzeiro, C L\'eonard, JC Zambrini , 
An entropic interpolation problem for incompressible viscid fluids, https://hal.archives-ouvertes.fr/hal-02179693.

\bibitem{Arnold} V. I. Arnold, Sur la g\'eom\'etrie diff\'erentielle des groupes de Lie de dimension infinie et ses applications
 \`a l' hydrodynamique des fluides parfaits,  \emph{Ann. Inst. Fourier}, 16 (1966), 316--361.

\bibitem{Bismut} J.M. Bismut, \emph{M\'ecanique al\'eatoire},  {Lect. Notes in Maths}, 866, Springer-Verlag, 1981.

\bibitem{Chan} C.H. Chan, M. Czubak, M.M. Disconzi, 
 The formulation of the Navier-Stokes equations on Riemannian manifolds,
 \emph{J. Geom. Phys.}, 121 (2017), 335--346.
 
 \bibitem{Chorin}  A.J. Chorin,  Numerical study of slightly visous flow. \emph{J. Fluid Mech.} 57 (1973), 785--796.
 
\bibitem{Cruzeiro} F. Cipriano, A.B. Cruzeiro, { Navier-Stokes equations and diffusions on the group of homeomorphisms
of the torus}, \emph{Comm. Math. Phys.}  275 (2007), 255--269.
 
 \bibitem{Constantin} P. Constantin, G. Iyer, A stochastic Lagrangian representation of the three-dimensional incompressible Navier-Stokes equations, \emph{ Comm. Pure Appl. Math.}, 61 (2008), 330--345.

\bibitem{EM} D.G. Ebin, J.E. Marsden, Groups of diffeomorphisms and the motion of an incompressible fluid,
\emph{Ann. of Math.} 92 (1970), 102--163.

\bibitem{Elworthy} K.D. Elworthy, \emph{Stochastic differential equations on manifolds}, Lond. Math. Soc. Lect. Note, 70, Cambridge 
university Press, 1982.

\bibitem{ELL} K.D. Elworthy,  Y. Le Jan, X.M. Li, \emph{On the geometry of diffusion operators and stochastic flows}, 
{Lecture Notes in Mathematics}, 1720, Springer-Verlag, 1999.

\bibitem{FangLuo}  S. Fang, D. Luo,  Constantin and Iyer's representation formula for the Navier-Stokes equations on manifolds,
\emph{Potential Analysis},  48 (2018), 181--206.

\bibitem{Flandoli} B. Busnello, F. Flandoli, M. Romito, A probabilistic representation for the vorticity
 of a three- dimensional viscous fluid and for general systems of parabolic equations.
  \emph{Proc. Edinb. Math. Soc.}  48 (2005),  295--336.
  
 \bibitem{IW} N. Ikeda, S. Watanabe, \emph{Stochastic differential equations and diffusion processes}, North-Holland, Math. Library,
 24, 1981.

\bibitem{Kobayashi}, M. H. Kobayashi, On the Navier-Stokes equations on manifolds with curvature, 
\emph{J Eng Math},  (2008) 60:55--68

\bibitem{KN} S. Kobayashi, K. Nomuzu, \emph{Foundations of differential geometry}, vol II, A Wiley-Interscience Publication, Wiley, New Yprk, 1962.

\bibitem{Kunita} H. Kunita, Stochastic differential equations and stochastic flows of diffeomorphisms. 
\emph{\'Ecole d'\'et\'e de Probabilit\'es de Saint-Flour, XII--1982}, 143--303, Lecture Notes in Math., 1097, Springer, Berlin, 1984.

\bibitem{Luo} D. Luo,  Stochastic Lagrangian flows on the group of volume-preserving homeomorphisms of the spheres. 
\emph{Stochastics} 87 (2015), 680--701.

\bibitem{Malliavin} P. Malliavin, Formule de la moyenne, calcul des perturbations et th\'eorie d'annulation pour les formes harmoniques,
\emph{J. Funt. Analysis}, 17 (1974), 274-291.

\bibitem{MT} M. Mitrea, M. Taylor, Navier-Stokes equations on Lipschitz domains in Riemannian manifolds, \emph{Math. Ann.}, 321(2001), 955--987.

\bibitem{Nagasawa} T. Nagasawa, Navier-Stokes flow on Riemannian manifolds,  \emph{Nonlinear Analysis theory, Method and Applications},
30 (1997), 825-832.

\bibitem{Pier} V. Pierfelice, The incompressible Navier-Stokes equations on non-compact manifolds, \emph{J. Geom. Anal.}, 27 (2017), 577--617.

\bibitem{Stroock} D. Stroock, \emph{An introduction to the analysis of paths on a Riemannian manifold}, Mathematical Surveys and Monographs, 74. {American Mathematical Society, Providence, RI}, 2000.

\bibitem{Taylor} M. Taylor,  \emph{Partial Differential Equations III: Nonlinear Equations, Nonlinear equations}, 
Vol. 117, Applied Mathematical Sciences, Springer New York second edition (2011).

\bibitem{Teman}  R. Teman,   \emph{Navier-Stokes equations and nonlinear functional analysis}, Second edition. CBMS-NSF Regional Conference Series in Applied Mathematics, 66, {Society for Industrial and Applied Mathematics (SIAM), Philadelphia, PA}, 1995.

\bibitem{TemamW} R. Temam, S. Wang, Inertial forms of Navier-Stokes equations on the sphere, \emph{J. Funct. Analysis},
117 (1993), 215--242.

\bibitem{Yano}  K. Yano, on harmonic and Killing vector fields, \emph{Ann. Math.}, 55 (1952), 38-45.

\bibitem{Zhang} X. Zhang, A stochastic representation for backward incompressible Navier--Stokes equations, \emph{Probab. Theory Related Fields} 148 (2010), no. 1--2, 305--332.

\end{thebibliography}
\end{document}